\documentclass[11pt]{amsart}
\usepackage[pdftex,colorlinks]{hyperref}
\usepackage{amsmath}
\usepackage{epsfig}
\usepackage{amsthm}
\usepackage{amssymb}
\usepackage{euscript}

\DeclareFontFamily{U}{mathx}{\hyphenchar\font45}
\DeclareFontShape{U}{mathx}{m}{n}{
      <5> <6> <7> <8> <9> <10>
      <10.95> <12> <14.4> <17.28> <20.74> <24.88>
      mathx10
      }{}
\DeclareSymbolFont{mathx}{U}{mathx}{m}{n}
\DeclareMathAccent{\widecheck}    {0}{mathx}{"71}
\numberwithin{equation}{section}
\theoremstyle{plain}

\newtheorem{theorem}{Theorem}[section]
\newtheorem{lemma}{Lemma}[section]

\newtheorem{remark}{Remark} [section]

\usepackage{amsmath, amsthm, amssymb, latexsym, mathrsfs}

\begin{document}
\title[Discrete Hardy-Littlewood-Sobolev Inequality]{Existence of the maximizing pair for the discrete Hardy-Littlewood-Sobolev
inequality}
\author[Huang]{Genggeng Huang}
\address{Department of Mathematics\\
INS and MOE-LSC\\
Shanghai Jiao Tong University, Shanghai} \email{genggenghuang@sjtu.edu.cn}
\author[Li]{Congming Li}
\address{Department of Mathematics\\
INS and MOE-LSC\\
Shanghai Jiao Tong University, Shanghai}\address{Department of Applied Mathematics, University of Colorado at Boulder}
\email{Congming.Li@Colorado.EDU}

\author[Yin]{Ximing Yin }
\address{Department of Mathematics\\
INS and MOE-LSC\\
Shanghai Jiao Tong University, Shanghai}
\email{jasonpkbl@sjtu.edu.cn}
\thanks{The work of C. Li is partially supported by NSF grant DMS-0908097 and NSFC grant 11271166}
\begin{abstract}
In this paper, we study the best constant of the following discrete Hardy-Littlewood-Sobolev inequality, \begin{equation}
\sum_{i,j,i\neq j}\frac{f_{i}g_{j}}{\mid i-j\mid^{n-\alpha}}\leq C_{r,s,\alpha} |f|_{l^r} |g|_{l^s},
\end{equation}where $i,j\in \mathbb Z^n$, $r,s>1$, $0<\alpha<n$, and $\frac 1r+\frac 1s+\frac {n-\alpha}n\geq 2$. Indeed, we can prove that the best constant is attainable in the supercritical case $\frac 1r+\frac 1s+\frac {n-\alpha}n> 2$.
\end{abstract}

\maketitle
\section{Introduction}
In the present paper, we investigate the attainability of the best constant of the following discrete Hardy-Littlewood-Sobolev(DHLS for abbreviation) inequality \begin{equation}
\label{DHLS1}\sum_{i,j,i\neq j}\frac{f_{i}g_{j}}{\mid i-j\mid^{n-\alpha}}\leq C_{r,s,\alpha} |f|_{l^r} |g|_{l^s},
\end{equation}where $i,j\in \mathbb Z^n$, $r,s>1$, $0<\alpha<n$, and $\frac 1r+\frac 1s+\frac {n-\alpha}n\geq 2$.
In fact, DHLS inequality is direct related to the classical Hardy-Littlewood-Sobolev(HLS) inequality\begin{equation}
\label{HLS1}\int_{\mathbb{R}^n}\int_{\mathbb{R}^n}\frac{f(x)g(y)}{\mid x-y\mid^{n-\alpha}}dxdy\leq C'_{r,s,\alpha}\| f\|_{L^r}\|g\|_{L^s}
\end{equation} for any $f\in L^r(\mathbb{R}^n)$ and $g\in L^s(\mathbb{R}^n)$ provided that
  \begin{displaymath}
  0<\alpha<n,1<r,s<\infty
  \end{displaymath}
  with
  \begin{displaymath}
  \frac{1}{r}+\frac{1}{s}+\frac{n-\alpha}{n}=2
  \end{displaymath}
  $C'_{r,s,\alpha}$ is the best constant for \eqref{HLS1}.
  \par  We now provide a proof from \eqref{HLS1} to get \eqref{DHLS1}. One may consider the special case of \eqref{HLS1} for $$f(x)\equiv f_i, g(x)\equiv g_i, \text{ in } |x-i|<\frac 1n,\forall i\in \mathbb Z^n,\text{ otherwise } f(x)=g(x)=0.$$
Obviously, we have \begin{eqnarray}&&\int_{\mathbb{R}^n}\int_{\mathbb{R}^n}\frac{|f(x)g(y)|}{\mid x-y\mid^{n-\alpha}}dxdy\nonumber\\&=&\sum_{i,j\in \mathbb Z^n}\int_{B_{\frac 1n}(i)}\int_{B_{\frac 1n}(j)} \frac{|f_i| |g_j|}{|x-y|^{n-\alpha}}dxdy\nonumber\\&>&\sum_{i,j,i\neq j}\frac{|f_i||g_j|}{|i-j|^{n-\alpha}}\int_{B_{\frac 1n}(0)}\int_{B_{\frac 1n}(0)} \frac{1}{|\frac{i-j+x-y}{|i-j|}|^{n-\alpha}}dxdy\nonumber\\
 &\geq &\sum_{i,j,i\neq j}\frac{|f_i||g_j|}{|i-j|^{n-\alpha}}\int_{B_{\frac 1n}(0)}\int_{B_{\frac 1n}(0)} \frac{1}{|1-\frac 1{\sqrt n}|^{n-\alpha}}dxdy
 \nonumber\\&\geq &c_n\sum_{i\neq j}\frac{|f_i||g_j|}{|i-j|^{n-\alpha}}
 \end{eqnarray}    Then by \eqref{HLS1}, we get  \eqref{DHLS1} immediately for $\frac 1r+\frac 1s+\frac {n-\alpha}n=2$. For the supercritical situation, we will present a simple lemma in Section 2 to illustrate it.
 \par It is well-known that  \eqref{HLS1} was studied by a remarkable paper of Lieb \cite{Lieb1983}. In \cite{Lieb1983}, Lieb proved the existence of the maximizing pair $(f,g)$, i.e. the attainability of the best constant of \eqref{HLS1}. In particular, Lieb also gave the explicit $(f,g)$ and $C'_{r,s,\alpha}$ in the case $p=q$. The method Lieb used was to examine the Euler-Lagrange equation that the maximizing pair $(f,g)$ satisfies. Also, we will analysis the Euler-Lagrange equation corresponding to  \eqref{DHLS1}. After Lieb \cite{Lieb1983}, Stein and Weiss first completed Lieb's work for weighted HLS inequality. There are also many other works concerning the Eluer-Lagrange equations corresponding to HLS inequality, see \cite{ChenJinLiLim}-\cite{ChenLiOu2006}.
\par  Now we turn to the discrete situation. For n=1,  \eqref{DHLS1} is just the Hardy-Littlewood-P\'olya (HLP) inequality \cite{Polya}.
In \cite{LiJohn}, the authors considered  \eqref{DHLS1} in a finite form under the assumptions that $r=s=2,\alpha=0$,\begin{equation}
\label{DHLS2}\sum_{i,j=1,i\neq j}^N \frac{f_i g_j}{|i-j|}\leq \lambda_N |f|_{2}|g|_{2}.
\end{equation} As this is a finite summation, \eqref{DHLS2} always holds by H\"older inequality for some constant $\lambda_N$ depending on $N$. From \eqref{DHLS1}, one can see that \eqref{DHLS2} fails for a uniform bound as $N\rightarrow \infty$. They proved that \begin{equation*}
\lambda_N=2\ln N+O(1).
\end{equation*} Recently, Cheng-Li \cite{ChengLi2013} generalized this result to high dimension for $r=s=2,\alpha=0$. They pointed out that the best constant $\lambda_N$ satisfied\begin{equation*}
\lambda_N=|S^{n-1}|\ln N+ o(\ln N),
\end{equation*} here $|S^{n-1}|$ represents the Lebesgue measure of the $n-1$ dimensional unit sphere. The regularities of the maximizing pair $(f,g)$ are also important in analysis. Chen-Li-Zhen \cite{ChenLiZhen2013} use the regularity lifting theorem obtained in \cite{ChenJinLiLim} to get the optimal summation interval of the solution of the Euler-Lagrange equation of  \eqref{DHLS1}. They also get some non-existence results.
 \par In our paper, we have the following theorem.
 \begin{theorem}\label{mainthm1}
 If $r,s>1,\alpha\in (0,n),\frac 1r+\frac 1s+\frac{n-\alpha}n>2$, then the best constant $C_{r,s,\alpha}$ for DHLS inequality \eqref{DHLS1} is attainable.
 \end{theorem}

 \begin{remark}
 In fact, the assumptions of  DHLS inequality \eqref{DHLS1} derived from HLS inequality \eqref{HLS1} should be $\frac 1r+\frac 1s+\frac {n-\alpha}n=2$. Later on, we will give a simple lemma to verify that \eqref{DHLS1} still holds for   $\frac 1r+\frac 1s+\frac {n-\alpha}n\geq 2$.
 \end{remark}
 \begin{remark}
 In the above theorem, we only proved the existence of the maximizing pair $(f,g)$ in the supercritical case $\frac 1r+\frac 1s+\frac {n-\alpha}n>2$.  But, we believe it is also valid for the critical case $\frac 1r+\frac 1s+\frac {n-\alpha}n= 2$.
 \end{remark}
 The main idea to prove Theorem \ref{mainthm1} is to consider a sequence of DHLS inequalities with finite elements as follows,\begin{equation}
 \label{DHLS3}\sum_{|i|\leq N}\sum_{|j|\leq N,i\neq j}\frac{f_i g_j}{|i-j|^{n-\alpha}}\leq C_{r,s,\alpha,N}|f|_r|g|_s,
 \end{equation} here $f=(f_i)_{|i|\leq N}$, $r,s>1,\frac 1r+\frac 1s+\frac{n-\alpha}n>2$. It is easy to see that \eqref{DHLS3} is the restriction of \eqref{DHLS1}  on $f,g$ with $f_i,g_i\equiv 0$ for $ |i| > N$. For later use, we denote $J(f,g)$ by $$J(f,g)=\sum_{i\in \mathbb Z^n}\sum_{j\in \mathbb Z^n,i\neq j}\frac{f_i g_j}{|i-j|^{n-\alpha}}.$$ Also we take $f^N, g^N$ with $|f^N|_r=|g^N|_s=1$  satisfy that
 $$J(f^N,g^N)=C_{r,s,\alpha,N}.$$ We want to prove $f^N,g^N\rightarrow f, g$ strongly in $l^r, l^s$ respectively. If it is right, we have proved Theorem \ref{mainthm1}. Unfortunately, this is always false as we can see DHLS inequality \eqref{DHLS1} is invariant under translation. We should use the Concentration Compactness ideas introduced by P.L. Lions. The following theorem is important for using Concentration Compactness ideas.
 \begin{theorem}
 \label{mainthm2} If $r,s>1,\alpha\in (0,n),\frac 1r+\frac 1s+\frac{n-\alpha}n>2$, then$$
 \max_{|i|\leq N} f_i^N,\max_{|i|\leq N} g_i^N\geq c_{r,s,\alpha}>0,$$ here $c_{r,s,\alpha}$ is a uniform constant  independent of $N$.
 \end{theorem}
 Theorem \ref{mainthm2} tells us that after a translation $\bar f^N_i=f^N_{i-i_0}$, we will have $\bar f^N_0=\displaystyle\max_{|i|\leq N} f^N=f^N_{i_0}$.
 This excludes the case $\bar f^N,\bar g^N\rightarrow 0$. We will have after translation, \begin{theorem}
 \label{mainthm3}Let $\bar f^N,\bar g^N$ be the translation of $f^N,g^N$, then  $J(\bar f^N,\bar g^N)\rightarrow C_{r,s,\alpha}$ and $\bar f^N, \bar g^N\rightarrow f, g$ strongly respectively in $l^r,l^s$ as $N\rightarrow \infty$.
 \end{theorem}
The present paper is organized as follows. In Section 2, we will prove Theorem \ref{mainthm2} and the first part of Theorem \ref{mainthm3}. This is the main part of this paper and the Concentration Compactness ideas is used in this section. We will prove Theorem \ref{mainthm1} in Section 3 and the last part of Theorem \ref{mainthm3} in Section 4.
\section{Concentration Compactness property}
\label{sec:1}
\setcounter{section}{2} \setcounter{equation}{0}
This section is devoted to prove Theorem \ref{mainthm2}.
First we shall illustrate  \eqref{DHLS1} with the following lemma.
\begin{lemma}
\label{lemcc0}Suppose $a\in l^p(\mathbb Z^n)$, then $|a|_{l^q}\leq |a|_{l^p}$ for $\forall q\geq p$.
\end{lemma}
\begin{proof}
For simplicity, we may assume $|a|_{l^p}=1$ which means $|a_i|\leq 1$, $i\in \mathbb Z^n$. This implies that
$$\sum_{i\in \mathbb Z^n}|a_i|^q\leq \sum_{i\in \mathbb Z^n}|a_i|^p=1.$$ This ends the proof of  the present lemma.
\end{proof}
By Lemma \ref{lemcc0}, one can directly get \eqref{DHLS1} from the critical case $\frac 1r+\frac 1s+\frac{n-\alpha}n=2$.
  \par  A directly computation easily yields  the Euler-Lagrange equation for \eqref{DHLS3}:
   \begin{equation}\begin{cases}\label{cc1}
   r(f_i^N)^{r-1}=\displaystyle \lambda \sum_{j\neq i}\frac{g_j^N}{|i-j|^{n-\alpha}}\\
   s(g_i^N)^{s-1}=\displaystyle \mu \sum_{j\neq i}\frac{f_j^N}{|i-j|^{n-\alpha}}
   \end{cases}\end{equation}
   If we multiply the first equation of \eqref{cc1} by $f_i^N$, the second equation by $g_i^N$ and sum up both sides, we can find out that
   $\frac{r}{\lambda}=\frac{s}{\mu}=C_{r,s,\alpha,N}$. From the definition of $C_{r,s,\alpha,N}$, it is easy to see that $C_{r,s,\alpha,N}>0$ is non-decreasing with respect to $N$. Moreover, we have the following lemma which corresponds to the first part of Theorem \ref{mainthm3}.
\begin{lemma}\label{lemcc}Let $C_{r,s,\alpha}$ and $C_{r,s,\alpha,N}$ be defined as in \eqref{DHLS1} and \eqref{DHLS3} respectively. We have\begin{equation*}
\lim_{N\rightarrow\infty} C_{r,s,\alpha,N}=C_{r,s,\alpha}
\end{equation*}
\end{lemma}
\begin{proof}
It is obvious that $$\lim_{N\rightarrow\infty} C_{r,s,\alpha,N}\leq C_{r,s,\alpha}.$$ Now we choose a maximizing sequence $f^{(m)},g^{(m)}>0$ with $|f^{(m)}|_{l^r}=|g^{(m)}|_{l^s}=1$ such that \begin{equation*}
\sum_{i\in \mathbb Z^n}\sum_{j\in \mathbb Z^n,i\neq j}\frac{f^{(m)}_i g^{(m)}_j}{|i-j|^{n-\alpha}}\geq C_{r,s,\alpha}(1-\frac 1m).
\end{equation*}Then we can choose $N_m$ large enough depending on $m$ such that $$\sum_{i\in \mathbb Z^n}\sum_{j\in \mathbb Z^n,i\neq j}\frac{f^{(m),N_m}_i g^{(m),N_m}_j}{|i-j|^{n-\alpha}}\geq C_{r,s,\alpha}(1-\frac 1m)^2.$$ Here $f^{(m),N_m}_i$ means that \begin{equation*}
f^{(m),N_m}_i=\begin{cases}
f^{(m)}_i, \quad &\text{for } |i|\leq N_m\\
0,&\text{for } |i|> N_m.
\end{cases}
\end{equation*} From the cut-off above, we have $$|f^{(m),N_m}_i|_{l^r}, |g^{(m),N_m}|_{l^s}\leq 1,C_{r,s,\alpha,N_m}\geq C_{r,s,\alpha}(1-\frac 1m)^2.$$ Passing $m\rightarrow \infty$, we get the desired result.
\end{proof}
By Lemma \ref{lemcc}, it is true that $$0<c_0\leq C_{r,s,\alpha,N}\leq C_0<\infty$$ for some uniform constants $c_0,C_0$.  Therefore without loss of generality, we may assume $C_{r,s,\alpha,N}=1$ in the proof of Theorem \ref{mainthm2}, since we only use the uniform up bound and lower bound of $C_{r,s,\alpha,N}$.
\par \textbf{The proof for Theorem \ref{mainthm2}}: Taking the equation of $f^N$ for instance, by \eqref{cc1} we have\begin{equation}\label{cc2}
(f_i^N)^{r-1}=\displaystyle  \sum_{|j|\leq N,j\neq i}\frac{g_j^N}{|i-j|^{n-\alpha}}\leq \max_{|j|\leq N} (g^N_j)^\epsilon  \sum_{|j|\leq N,j\neq i}\frac{(g_j^N)^{1-\epsilon}}{|i-j|^{n-\alpha}}.
\end{equation}Here $0<\epsilon<1$ is a parameter to be determined later.
This means that \begin{equation}\label{cc3}1=\sum_{|i|\leq N}(f_i^N)^{r}\leq \max_{|k|\leq N} (g^N_k)^{\frac{r\epsilon}{1-r}} \sum_{|i|\leq N}\left( \sum_{|j|\leq N,j\neq i}\frac{(g_j^N)^{1-\epsilon}}{|i-j|^{n-\alpha}}\right)^{\frac r{r-1}}\end{equation}
Now we define an operator $T$ satisfying:
\begin{equation*}
(Tf)_i=\sum_{j\in \mathbb Z^n,j\neq i}\frac{f_i}{\mid j-i\mid^{n-\alpha}}.
\end{equation*}
Then by DHLS inequality, we have
$$|Tf|_{l^p}\leq C|f|_{l^q}$$ for $\frac 1q+\frac{n-\alpha}n=1+\frac 1p$. We take $p=\frac r{r-1},q=\frac{s}{1-\epsilon}$, then we can get the righthand side of \eqref{cc3},\begin{eqnarray}\label{cc4}
|T((g^N)^{1-\epsilon})|_{l^{\frac r{r-1}}}\leq C|(g^N)^{1-\epsilon}|_{l^{\frac s{1-\epsilon}}}=C.
\end{eqnarray}To guarantee \eqref{cc4}, we need $$\frac{r-1}{r}+1=\frac{1-\epsilon}s+\frac{n-\alpha}n,\quad \text{i.e.},\quad 2+\frac \epsilon s=\frac 1s+\frac 1r+\frac{n-\alpha}n.$$ By the assumption of Theorem \ref{mainthm2}, we see $\epsilon>0$. Also, as $\frac 1r+\frac{n-\alpha}n<2$, we must have $\epsilon<1$. From \eqref{cc4}, one can get $$\max_{|k|\leq N} (g^N_k)^{\frac{r\epsilon}{1-r}} C^{\frac r{r-1}}\geq 1.$$ Or we have
$$\max_{|k|\leq N} (g^N_k)\geq c_{r,s,\alpha}.$$ The proof for the lower bound of $\displaystyle\max_{|k|\leq N} (f^N_k)$ is just the same, we omit the details here.$\Box$
\begin{remark}
From the proof of Theorem \ref{mainthm2}, one can see that the supercritical condition $\frac 1r+\frac 1s+\frac {n-\alpha}n>2$ plays an important role. This is also the reason why we can't prove  the existence of the maximizing pair $(f,g)$ in the critical case in the present paper.
\end{remark}
\section{The existence of maximizing pair $(f,g)$}
\setcounter{section}{3} \setcounter{equation}{0}
Since Theorem \ref{mainthm2}, we can have \begin{equation*}
f^{(N)}_i=\begin{cases}\bar f_i^N,\quad&\text{in}\quad \Omega_N\\
0,&\text{in}\quad\mathbb Z^n\backslash \Omega_N\end{cases}
\end{equation*} Here $\Omega_N=\{i-i_0||i|\leq N\}$,  $(f^N,g^N)$ is the maximizing pair of \eqref{DHLS3} and $\bar f_i^N=f^N_{i-i_0}$ with $\displaystyle \max_{|i|\leq N} f^N_i=f^N_{i_0}$. One can see $f^{(N)}_0,g^{(N)}_0\geq c>0$.
As $|f^{(N)}|_{l^r}=|g^{(N)}|_{l^s}=1$, we can choose a subsequence still denoted by $f^{(N)},g^{(N)}$ such that$$
f^{(N)}\rightharpoonup f, g^{(N)}\rightharpoonup g, \quad \text{ weakly in}\quad l^r,l^s\quad \text{respectively}$$ and $$f^{(N)}\rightarrow f,g^{(N)}\rightarrow g, \quad \text{pointwise in }\mathbb Z^n.$$ It is easy  to see that $f_0,g_0\geq c>0$ and $|f|_{l^r},|g|_{l^s}\leq 1$. Now we can have the following lemma.
\begin{lemma}\label{lemext1}$\forall i,j\in \mathbb Z^n$, $f_i>0,g_j>0$, we have \begin{equation}\label{ext1}
\begin{cases}
\displaystyle C_{r,s,\alpha}f_i^{r-1}=\sum_{k,k\neq i}\frac{g_k}{|k-i|^{n-\alpha}}\\
\displaystyle C_{r,s,\alpha}g_j^{s-1}=\sum_{k,k\neq i}\frac{f_k}{|k-j|^{n-\alpha}}.
\end{cases}\end{equation}
\end{lemma}
\begin{proof}
We only need to show the first part of  \eqref{ext1} is right. As $f_i>0$, we can see $f^{(N)}_i>0$ for $N$ large, then  for any fixed $M$,\begin{equation}\label{ext2}
\displaystyle C_{r,s,\alpha,N}\left(f_i^{(N)}\right)^{r-1}=\sum_{k,k\neq i}\frac{g_k^{(N)}}{|k-i|^{n-\alpha}}=\sum_{|k|\leq M,k\neq i}\frac{g_k^{(N)}}{|k-i|^{n-\alpha}}+\sum_{|k|>M,k\neq i}\frac{g_k^{(N)}}{|k-i|^{n-\alpha}}
\end{equation} We can pass the limit in the left-hand side and the first part of right-hand side of \eqref{ext2} since it is finite summation.
 \begin{eqnarray}
 \sum_{|k|>M,k\neq i}\frac{g_k^{(N)}}{|k-i|^{n-\alpha}}&\leq&  \left(\sum_{|k|>M,k\neq i}\left(g_k^{(N)}\right)^s\right)^{\frac 1s}\left(\sum_{|k|>M,k\neq i} |k-i|^{-\frac{(n-\alpha)s}{s-1}}\right)^{\frac{s-1}s}\nonumber\\
 &\leq & C M^{n-\frac{(n-\alpha)s}{s-1}}\rightarrow 0, \quad \text{as }M  \rightarrow \infty.
  \end{eqnarray}In getting the last inequality, we have used $\frac 1r+\frac 1s+\frac {n-\alpha}n>2$ which means $\frac 1s>\frac\alpha n$.
\end{proof}

\begin{lemma}\label{lemext2}$$|f|_{l^r}=|g|_{l^s}=1.$$\end{lemma}
\begin{proof}
If it's not true, by Lemma \ref{lemext1}, we can easily see that $0<|f|_{l^r}^r=|g|_{l^s}^s<1$. Set $$\bar f_i=\frac{f_i}{|f|_{l^r}},\bar g_i=\frac{g_i}{|g|_{l^s}}.$$ Then \begin{eqnarray}
J(\bar f,\bar g)&=&\sum_i\sum_{j,j\neq i}\frac{f_i g_j}{|i-j|^{n-\alpha}}|f|_{l^r}^{-1}|g|_{l^s}^{-1}\nonumber\\
&=&C_{r,s,\alpha}|f|_{l^r}^{r-1-\frac rs}>C_{r,s,\alpha}
\end{eqnarray}which is a contradiction to the definition of best constant. The last inequality follows from$$
\frac 1r+\frac 1s>\frac \alpha n+1>1.$$
\end{proof}
In fact, Lemma \ref{lemext2} implies Theorem \ref{mainthm1} with $(f,g)$ as the maximizing pair.
Although in passing the limit to get  the maximizing pair $(f,g)$, we may only have $f_i,g_i>0$ for $i\in \Omega\subset \mathbb Z^n$. But in fact, as we know $(f,g)$ are maximizing pair, they should satisfy \eqref{ext1} for all $i\in \mathbb Z^n$ which means $f,g>0$.

\section{The strong convergence of $\bar f^N,\bar g^N$}
\setcounter{section}{4} \setcounter{equation}{0}
This section is denoted to prove the second part of Theorem \ref{mainthm3}. The following lemma is a special case of Theorem 2 in \cite{BrezisLieb1983}. We provide a simple proof here.
\begin{lemma}\label{lemstr}
Suppose  $f^N\in l^r(\mathbb Z^n)$ and $f^N\rightarrow f\in l^r(\mathbb Z^n)$ pointwise in $\mathbb Z^n$ as $N\rightarrow \infty$.  Then we will have $$ \lim_{N\rightarrow \infty}|f^N-f|_{l^r}=0,$$ provided $\displaystyle \lim_{N\rightarrow \infty}|f^N|_{l^r}=|f|_{l^r}$.
\end{lemma}
\begin{proof}
For any fixed $\epsilon>0$, we can choose $M$ large enough such \begin{equation}\label{str1}\sum_{|i|\leq M}|f_i|^r\geq |f|_{r}^r(1-\epsilon).\end{equation}Fox such fixed $M$, by the pointwise convergence of $f^N$, we can choose $N_M$ large enough so that \begin{equation}\label{str2}
\sum_{|i|\leq M}|f^N_i-f_i|^r\leq \epsilon|f|_r^r,\quad \forall N\geq N_M.
\end{equation} Combining \eqref{str1} and \eqref{str2}, we have $\displaystyle (1-2\epsilon)|f|_r^r\leq \sum_{|i|\leq M}|f^N_i|^r\leq |f|_r^r$. From $\displaystyle \lim_{N\rightarrow \infty}|f^N|_{l^r}=|f|_{l^r}$, we have another $N_\epsilon$ $$(1-\epsilon)|f|_r^r\leq |f^N|_r^r\leq (1+\epsilon)|f|_r^r, \quad \forall N\geq N_\epsilon. $$ Hence for $N\ge \max(N_M,N_\epsilon)=N_0$,
$$\sum_{|i|> M}|f^N_i|^r\leq 3\epsilon|f|_r^r.$$ Now we have \begin{eqnarray*}
|\bar f^N-f|_r^r&=&\sum_{|i|\leq M}|f^N_i-f_i|^r+\sum_{|i|> M}|f^N_i-f_i|^r\\
&\leq &\epsilon |f|_r^r+C\sum_{|i|> M}(|f^N_i|^r+|f_i|^r)\leq 5\epsilon |f|_r^r
\end{eqnarray*} for $N\geq N_0$. Passing $\epsilon\rightarrow 0$, we have finished the proof of the present lemma.
\end{proof}
The second part of Theorem \ref{mainthm3} is the direct conclusion of Lemma \ref{lemext2} and Lemma \ref{lemstr}.

%



\bigskip




\medskip
\medskip

\end{document}